\documentclass[12pt]{article}
\usepackage[top=2.5cm,bottom=2.5cm,left=2.9cm,right=2.9cm]{geometry}
\usepackage{amssymb}
\usepackage{amsmath,amsthm}
\usepackage[latin1]{inputenc}
\usepackage{graphicx}
\usepackage{hyperref}
\usepackage{enumerate}
\usepackage{tikz}
\usepackage{tkz-graph}
\usepackage{mathrsfs}
\usepackage{verbatim}
\usepackage{upgreek}
 \usepackage{mathptmx}

\usepackage{upgreek}
\usepackage{mathptmx}
\usepackage{eucal}

\usepackage{colortbl}

\hypersetup{colorlinks=true}

\hypersetup{colorlinks=true, linkcolor=blue, citecolor=blue,urlcolor=blue}

\newtheorem{remark}{Remark}[section]

\newtheorem{lemma}[remark]{Lemma}
\newtheorem{theorem}[remark]{Theorem}
\newtheorem{proposition}[remark]{Proposition}

\newtheorem{corollary}[remark]{Corollary}

\newtheorem{conjecture}[remark]{Conjecture}

\title{Protection of graphs with emphasis on Cartesian product graphs}
\author{Magdalena Valveny, Juan Alberto Rodr\'{\i}guez-Vel\'{a}zquez\\
{\small Universitat Rovira i Virgili }\\{\small Departament d'Enginyeria Inform\`atica i Matem\`atiques } \\  {\small Av. Pa\"{\i}sos Catalans 26, 43007 Tarragona, Spain.} \\{\small
 magdalena.valveny\@@estudiants.urv.cat, juanalberto.rodriguez\@@urv.cat}
}

\date{ }
\begin{document}
\maketitle

\begin{abstract}
In this paper we study the weak Roman domination number and the secure domination number of a graph. In particular, we obtain general bounds on these two parameters and, as a consequence of the study, we derive new inequalities of Nordhaus-Gaddum type involving secure domination and weak Roman domination. Furthermore,  the particular case of  Cartesian product graphs is considered. 
\end{abstract}



\section{Introduction}

The following approach to protection of a graph was described by Cockay\-ne et al.\  
\cite{MR2137919}. Suppose that one or more guards are stationed at some of the vertices of a simple graph $G$ and that a guard at a vertex can deal with a problem at any vertex in its closed neighbourhood. Consider a function $f: V(G)\longrightarrow \{0,1,2,\dots\}$ where $f(v)$  is the number of guards at $v$,  and let $V_i=\{v\in V(G):\; f(v)=i\}$ for every $i\in \{0,1,2,\dots\}$. We will identify $f$ with the partition of $V(G)$ induced by $f$ and write $f(V_0,V_1, \dots).$ The weight of $f$ is defined to be $w(f)=\sum_{v\in V(G)}f(v)=\sum_ii|V_i|$. A vertex $v\in V(G)$ is \textit{undefended} with respect to  $f$ if  $f(v)=0$ and $f(u)=0$ for every vertex $u$ adjacent to $v$. We say that $G$ is \textit{protected} under the function $f$  if  $f$ has no undefended vertices,   \emph{ i.e.}, $G$ is protected if there is at least one guard available to handle a problem at any vertex.  We now define the four particular subclasses of protected graphs considered in ~\cite{MR2137919}. The functions in each subclass protect the graph according to a certain strategy.

\begin{itemize}
\item We say that $f(V_0,V_1)$ is a \emph{dominating function} (DF) if $G$ is protected under~$f$. Obviously, $f(V_0,V_1)$  is a DF if and only if $V_1$ is a dominating set.  The \emph{domination number}, denoted by $\gamma(G)$ is the minimum cardinality among all dominating sets of $G$.
This method  of protection has been  studied extensively \cite{Haynes1998a,Haynes1998}.

\item A \textit{Roman dominating function} (RDF) is a  function $f(V_0,V_1,V_2)$ such that  for every $v\in V_0$ there exists a vertex $u\in V_2$ which is adjacent to $v$. The \textit{Roman domination number}, denoted by  $\gamma_R(G)$, is the minimum weight among all  Roman dominating functions on $G$.  
This concept of protection has historical motivation \cite{Stewart1999} and was formally proposed by Cockayne et al.\ in \cite{Cockayne2004}.

\item A \textit{weak Roman dominating function} (WRDF) is a  function $f(V_0,V_1,V_2)$ such that  for every $v$ with $f(v)=0$ there exists a vertex $u$ adjacent to $v$ such that $f(u)\in  \{1,2\}$ and the function $f': V(G)\longrightarrow \{0,1,2\}$ defined by $f'(v)=1$, $f'(u)=f(u)-1$ and $f'(z)=f(z)$ for every $z\in V(G)\setminus\{u,v\}$, has no undefended vertices. The \textit{weak Roman domination number}, denoted by  $\gamma_r(G)$, is the minimum weight among all weak Roman dominating functions on $G$.
A WRDF of weight  $\gamma_r(G)$ is called a $\gamma_r(G)$-function.  For instance, for the tree shown in Figure \ref{FigWeakRoman}, on the left,  a $\gamma_r(G)$-function can place $2$ guards at the vertex of degree three and one guard at the other black-coloured vertex. This concept of protection was introduced  by Henning and Hedetniemi \cite{MR1991720} and studied further in \cite{MR3258160,Cockayne2003,Valveny2017}.

\item  A \textit{secure dominating function} is a WRDF function $f(V_0,V_1,V_2)$ in which $V_2=\emptyset$. In this case, it is convenient to define this concept of save graph by the properties of $V_1$. Obviously $f(V_0,V_1)$ is a secure dominating function if and only if $V_1$ is a dominating set and for every $v\in V_0$ there exists $u\in V_1$ which  is adjacent to $v$ and  $(V_1\setminus \{u\})\cup \{v\}$ is a dominating set. In such a case, $V_1$ is said to be a \textit{secure dominating set}. The \emph{secure domination number}, denoted by $\gamma_s(G)$,  is the minimum cardinality among all secure dominating sets. A secure dominating  function of weight $\gamma_s(G)$ is called a $\gamma_s(G)$-function. Analogously, a secure dominating set of cardinality $\gamma_s(G)$ is called a $\gamma_s(G)$-set.
This concept of protection was introduced  by Cockayne et al.\ in \cite{MR2137919}, and studied further in \cite{MR3355313,MR2529132,MR3258160,Cockayne2003,MR2477230}.
\end{itemize}

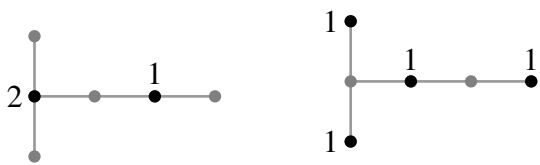
\begin{figure}[h]
\begin{center}
\begin{tikzpicture}
[line width=1pt, scale=0.8]

\coordinate (V1) at (1,0);
\coordinate (V2) at (2,0);
\coordinate (V3) at (3,0);
\coordinate (V4) at (4,0);
\coordinate (V5) at (1,-1);
\coordinate (V6) at (1,1);

\draw[black!40]  (V1)--(V2)--(V3)--(V4);
\draw[black!40]  (V1)--(V5);
\draw[black!40]  (V1)--(V6);

\foreach \number in {1,...,6}{
\filldraw[gray]  (V\number) circle (0.08cm);
}
\foreach \number in {1,3}{
\filldraw[black]  (V\number) circle (0.08cm);
}

\node [left] at (1,0) {$2$};
\node [above] at (3,0) {$1$};
\end{tikzpicture}
\hspace{1cm}
\begin{tikzpicture}
[line width=1pt, scale=0.8]

\coordinate (V1) at (1,0);
\coordinate (V2) at (2,0);
\coordinate (V3) at (3,0);
\coordinate (V4) at (4,0);
\coordinate (V5) at (1,-1);
\coordinate (V6) at (1,1);

\draw[black!40]  (V1)--(V2)--(V3)--(V4);
\draw[black!40]  (V1)--(V5);
\draw[black!40]  (V1)--(V6);

\foreach \number in {1,...,6}{
\filldraw[gray]  (V\number) circle (0.08cm);
}
\foreach \number in {2,4,5,6}{
\filldraw[black]  (V\number) circle (0.08cm);
}

\node [left] at (1,1) {$1$};
\node [left] at (1,-1) {$1$};
\node [above] at (4,0) {$1$};
\node [above] at (2,0) {$1$};
\end{tikzpicture}
\end{center}
\caption{Two placements of guards which correspond  to  two different weak Roman dominating functions on the same tree. Notice that $2=\gamma(G)<\gamma_r(G)<\gamma_s(G)=4$.}
\label{FigWeakRoman} 
\end{figure}

The problem of computing $\gamma_r(G)$ is NP-hard, even when restricted to bipartite or chordal graphs \cite{MR1991720}, and the problem of computing $\gamma_s(G)$ is also NP-hard, even when restricted to split graphs \cite{MR3355313}. This suggests finding
the weak Roman domination number and the secure domination number  for special classes of graphs or obtaining good bounds on these invariants.  This is precisely the aim of this work.   The remainder of the paper is structured as follows. Section \ref{SectionBasicWRDN} is devoted to obtain general bound on $\gamma_r(G)$ and $\gamma_s(G)$ in terms of several invariants of $G$. As a consequence of the study we derive new inequalities of Nordhaus-Gaddum type involving secure domination and weak Roman domination. Finally, in Section \ref{SectionCartesianProduct} we restrict our study to the particular case of  Cartesian product graphs.

Throughout the paper, we will use the notation $K_t$, $K_{1,t-1}$, $C_t$, $N_t$ and $P_t$ for complete graphs, star graphs, cycle graphs, empty graphs and path graphs of order $t$, respectively. We use the notation  $G \cong H$ if $G$ and $H$ are isomorphic graphs. For a vertex $v$ of a graph $G$, $N(v)$ will denote the set of neighbours or \emph{open neighbourhood} of $v$ in $G$. The \emph{closed neighbourhood}, denoted by $N[v]$, equals $N(v) \cup \{v\}$.  
We denote by $\delta(v)=|N(v)|$ the degree of vertex $v$, as well as $\delta(G) =\min_{v \in V(G)}\{\delta(v)\}$,  $\Delta(G) =\max_{v \in V(G)}\{\delta(v)\}$ and $n(G)=|V(G)|$.    For the remainder of the paper, definitions will be introduced whenever a concept is needed.

\section{General bounds}\label{SectionBasicWRDN}

To begin this section we would emphasize the following inequality chains.

\begin{proposition}{\rm \cite{MR2137919}} \label{DominationChain1}
The following inequalities hold for any graph $G$.
 \begin{enumerate}[{\rm (i)}]
\item $\gamma(G) \le \gamma_r(G)\le \gamma_R(G)\le 2\gamma(G).$
 
 \item $\gamma(G) \le \gamma_r(G)\le \gamma_s(G).$
 \end{enumerate}
\end{proposition}

The problem of characterizing the graphs with $\gamma_r(G)=\gamma(G)$ was solved by Henning and Hedetniemi \cite{MR1991720}.
The inequality chain (ii) has motivated us to obtain the following result, which shows that the problem of characterizing the graphs with $\gamma_s(G)=\gamma(G)$ is already solved.

\begin{theorem}Let $G$ be a graph. The following statements are equivalent.
\begin{enumerate}[{\rm (i)}]
\item $\gamma_r(G)=\gamma(G)$.
\item $\gamma_s(G)=\gamma(G)$.
\end{enumerate}
\end{theorem}

\begin{proof}
By Proposition \ref{DominationChain1} (ii), $\gamma_s(G)=\gamma(G)$ leads to  $\gamma_r(G)=\gamma(G)$. Now, if $\gamma_r(G)=\gamma(G)$, then for  any $\gamma_r(G)$-function $f(V_0,V_1,V_2)$ we have $V_2=\emptyset $, as $V_1\cup V_2$ is a dominating set and $\gamma(G)=\gamma_r(G)=|V_1|+2|V_2|\ge |V_1|+|V_2|\ge \gamma(G)$. Hence, $V_1$ is a secure dominating set, which implies that  $\gamma(G)=|V_1|\ge \gamma_s(G)\ge\gamma(G)$. Therefore, $\gamma_s(G)=\gamma(G)$.
\end{proof}

Given a graph $G$ and an edge $e\in E(G)$, the graph obtained from $G$ by removing  $e$ will be denoted by $G-e$, \textit{i.e}.,  $V(G-e)=V(G)$ and $E(G-e)=E(G)\setminus \{e\}.$ As observed in \cite{MR1991720}, any $\gamma_r(G-e)$-function is a WRDF for $G$. Similarly, any $\gamma_s(G-e)$-set is a secure dominating set for $G$.   
Therefore,  the following basic result follows.

\begin{proposition}\label{RemovingEdges}
The following statement hold for any spanning subgraph $H$ of a graph $G$.

\begin{enumerate}[{\rm (i)}]
\item {\rm \cite{MR1991720}} $\gamma_r(G)\le \gamma_r(H).$ 
\item $\gamma_s(G)\le \gamma_s(H).$ 
\end{enumerate}
\end{proposition}

 \newpage 

\begin{proposition} \label{cycles and pathas}
For any integer $t\ge 4$,
\begin{enumerate}[{\rm (i)}]
\item{\rm \cite{MR1991720}}  $\gamma_r(C_t)=\gamma_r(P_t)=\left\lceil\frac{3t}{7}\right\rceil.$
\item {\rm \cite{MR2137919}}  $\gamma_s(C_t)=\gamma_s(P_t)=\left\lceil\frac{3t}{7}\right\rceil.$
\end{enumerate}
\end{proposition}

By Proposition \ref{RemovingEdges} (ii) and Proposition \ref{cycles and pathas} (ii) we deduce the following result.

\begin{theorem}\label{BoundSecureDominatingHamiltonia}
For any Hamiltonian graph $G$ with $n(G)\ge 4$,
$$\gamma_s(G)\le \left\lceil\frac{3n(G)}{7}\right\rceil.$$
\end{theorem}

Obviously, the bound above is tight, as it is achieved by any cycle graph of order at least four.

A set $S\subseteq V(G)$ is a $k$-\emph{dominating set} if $|N(v)\cap S|\ge k$ for every $v\in \overline{S}$. The minimum cardinality among all $k$-dominating sets is called the $k$-\emph{domination number} of $G$ and it is denoted by $\gamma_k(G)$. 
It is readily seen that any $2$-dominating set is a secure dominating set. Therefore, we can state the following result.

\begin{theorem}{\rm \cite{MR3258160}}\label{UpperBound2domination}
For any graph  $G$,
 $$\gamma_s(G)\le  \gamma_2(G).$$
\end{theorem}

\begin{theorem}{\rm \cite{MR2529132}}\label{UpperBoundnOver2}
Let $G\not\cong C_5$ be a  connected graph. If $\delta(G)\ge~2$, then $$\displaystyle\gamma_s(G)\le \left\lfloor \frac{n(G)}{2}\right\rfloor .$$
\end{theorem}
An example of a graph with $\delta(G)=3$ and $\gamma_s(G)=\gamma_2(G)= \left\lfloor \frac{n(G)}{2} \right\rfloor$ is the $3$-cube graph. Notice that from the result above and the fact that $\gamma_r(G)\le \gamma_s(G)$ we can conclude that if $G\not\cong C_5$ is connected and  $\delta(G)\ge 2$, then $ \displaystyle\gamma_r(G)\le \left\lfloor \frac{n(G)}{2}\right\rfloor .$ With the aim of providing a general upper bound on the weak Roman domination number of any graph in terms of $n(G)$, we need to introduce some additional notation. 
For any support vertex $v$ of a tree $T$, the set of leaves adjacent to $v$ in $T$ will 
be denoted by  $L_T(v)$.   Let $S(T)$ be the set of support vertices $v\in V(T)$ of degree  $\delta(v)\le |L_T(v)|+1$ and define
$$X(T)=\bigcup_{v\in S(T)}(\{v\}\cup L_{T}(v)) .$$ 
 Let $T_0,T_1,\dots, T_k$ be the sequence of all embedded subtrees of $T$, of order greater than or equal to three, defined as follows: $T_0=T$ and $T_i$ is the subtree of $T_{i-1}$ induced by $V(T_{i-1})\setminus X(T_{i-1})$, for every $i\in \{ 1,\dots ,k\}$. Notice that the smallest subtree $T_k$ satisfies  $|V(T_k)\setminus X(T_k)|\le 2$. With this notation in mind we proceed to prove the two following  results.

\begin{theorem}
For any connected nontrivial graph $G$,
 $$\gamma_r(G)\le \left\lfloor \frac{2n(G)}{3}\right \rfloor.$$
\end{theorem}

\begin{proof}
Since the case $n(G)=2$ is straightforward, we can assume that $n(G)\ge 3$. Let $T$ be a spanning tree of $G$ and $T_0,T_1,\dots, T_k$  the sequence of all embedded subtrees of $T$ of order greater than or equal to three defined previously. By Proposition \ref{RemovingEdges}, $\gamma_r(G)\le \gamma_r(T)$. It remains to show that $\gamma_r(T)\le \frac{2n(G)}{3}$. To this end, we proceed to construct a WRDF $f$ such that $w(f)\le \frac{2n(G)}{3}$. 
 
For every $v\in X(T_i)$ and $i\in \{0,\dots, k\}$ we set 
$$f(v)=\left\lbrace{\begin{array}{cl}
2&\text{ if }v\in S(T_i) \text{ and } |L_{T_i}(v)|\ge 2,
\\
\\
1&\text{ if }v\in S(T_i)  \text{ and } |L_{T_i}(v)|=1,
\\
\\
0&\text{ if }v\in X(T_i)\setminus S(T_i) .
\end{array}}\right.$$
Notice that $V(G)=\displaystyle\bigcup_{i=0}^kX(T_i)\cup\left( V(T_k)\setminus X(T_k)\right)$
and $X(T_i)\cap X(T_j)=\emptyset$ for every $i\ne j$. Hence,   it remains to define $f(x)$ for every $x\in V(T_k)\setminus X(T_k)$, if any. 

Notice that for any $i\in \{0,\dots, k\}$,  
\begin{equation}\label{eq2/3}
\sum_{v\in X(T_i)}f(v)=\sum_{v\in S(T_i)}f(v)\le \frac{2}{3}|X(T_i)|
\end{equation} and, if there is a support vertex $v$ of $T_i$ with $|L_{T_i}(v)|=1$, then \begin{equation}\label{lesseq2/3}
\sum_{v\in X(T_i)}f(v)=\sum_{v\in S(T_i)}f(v)<\frac{2}{3}|X(T_i)|.
\end{equation}

Hence, if 
$V(T_k)=X_{k}$  then  $\sum_{i=0}^{k} |X(T_i)|=n(G)$, which implies that
$$w(f)=\sum_{i=0}^{k}\left( \sum_{v\in X(T_i)}f(v)   \right) \le \frac{2}{3}\sum_{i=0}^{k} |X(T_i)|\le\frac{2n(G)}{3}.$$ 
Suppose that $V(T_k)\setminus X_{k}=\{x\}$. In this case, we set $f(x)=0$  whenever $f(v)=2$ for some neighbour $v$ of $x$,  otherwise we set $f(x)=1$. Obviously, if $f(x)=0$, then 
 $$w(f)=\sum_{i=0}^{k}\left( \sum_{v\in X(T_i)}f(v)   \right)+f(x) \le \frac{2}{3}\sum_{i=0}^{k} |X(T_i)|\le\frac{2(n(G)-1)}{3}< \frac{2n(G)}{3}.$$ 
 Now, if $f(x)=1$, then \eqref{lesseq2/3} leads to  $\sum_{v\in X_{k}}f(v)\le\frac{2}{3}|X_{k}|-1$, which implies that 
 \begin{align*}
 w(f)=&\sum_{i=0}^{k-1}\left( \sum_{v\in X(T_i)}f(v)   \right) +\sum_{v\in X_{k}}f(v)+f(x)\\
  \le &\frac{2}{3}\sum_{i=0}^{k-1}|X(T_i)|+\left(\frac{2}{3}|X(T_k)|-1\right)+1\\
  =& \frac{2(n(G)-1)}{3}< \frac{2n(G)}{3}.
 \end{align*}
Finally, if $V(T_k)\setminus X_{k}=\{a,b\}$ , then  we set $f(a)=0$ and $f(b)=1$. Thus,
\begin{align*}
w(f)&=\sum_{i=0}^{k}\left( \sum_{v\in X(T_i)}f(v)   \right)+f(a)+f(b) \\
&\le \frac{2}{3}\sum_{i=0}^{k} |X(T_i)|+1\\
&= \frac{2(n(G)-2)}{3}+1<\frac{2n(G)}{3}.
\end{align*}
 In summary, we can conclude that $w(f)\le \frac{2n(G)}{3}$, and it is readily seen that $f$ is a WRDF. Therefore, the result follows.
\end{proof}

To see that the bound above is tight we can take any graph $G_1$ and construct the corona graph $G\cong G_1\odot N_2$ by considering one copy of $G_1$ and $n(G_1)$ copies of $N_2$ and joining, by an edge, each vertex of $G_1$ with the vertices in the corresponding copy of $N_2$. In this case we have $\gamma_r(G)=2n(G_1)$ and $n(G)=3n(G_1)$.

\begin{theorem}\label{SuperBoundSecDomTrees}
Let $T$ be a spanning tree  of a connected graph $G$ such that $n(G)\ge 3$. If $T_0,T_1,\dots, T_k$ is the sequence of all embedded subtrees of $T$ of order greater than or equal to three defined above, then 
$$\gamma_s(G)\le \sum_{i=0}^k\sum_{v\in S(T_i)}|L_{T_i}(v)|+\varrho(T),$$
where  $\varrho(T)=0$ if $V(T_k)=X(T_k)$ and $\varrho(T)=1$ otherwise.
\end{theorem}

\begin{proof}
Notice that Proposition \ref{RemovingEdges} leads to $\gamma_s(G)\le \gamma_s(T)$. Let  $$W=\bigcup_{i=0}^k\left( \bigcup_{v\in S(T_i)}L_{T_i}(v) \right) \cup W_k,$$  
where  $W_k$ is defined as follows. If $V(T_k)=X(T_k)$, then we set $W_k=\emptyset $, otherwise we fix $x_k\in  V(T_k)\setminus X(T_k)$ and we set  $W_k=\{x_k\}$. To  conclude  that $W$ is a  secure dominating set for $T$  we only need to 
observe that $W$ is a dominating set and the movement of a guard from $L_{T_i}(v)$ to $v$ does not produce undefended vertices, as well as, the movement of a guard from $x_k$ to a vertex in $ V(T_k)\setminus X(T_k)$ (if any) does not produce undefended vertices. Therefore, the result follows.
\end{proof}

The bound above is achieved, for instance, by the family of corona graphs $G\cong G_1\odot N_t$. Obviously, for any spanning tree $T$ of $G$ we have $\varrho(T)=0$ and $tn(G_1)\le \gamma_r(G)=\sum_{i=0}^k\sum_{v\in S(T_i)}|L_{T_i}(v)|+\varrho(T)=tn(G_1)$. Notice that the lower bound $\gamma_r(G)\ge tn(G_1)$ is deduced from the fact that every secure dominating set contains at least one guard per each vertex of degree one in $G$. In general, we can state the following tight bound in terms of the number  of vertices of degree one, denoted by $\ell(G)$.

\begin{remark}
For any graph $G$, $$\gamma_s(G)\ge \ell(G).$$
In particular, for any graph $G'$, $$\gamma_s(G'\odot N_t)= \ell(G'\odot N_t)=n(G')t.$$
\end{remark}

Two edges in a graph $G$ are
independent if they are not adjacent in $G$. The \emph{matching number} $\alpha'(G)$ of graph $G$, sometimes known as the edge independence number, is the cardinality of a maximum independent edge set.


\begin{theorem} \label{alphaCockayne2003} {\rm\cite{Cockayne2003}}
If a graph $G$ does not have isolated vertices, then $$\gamma_s(G)\le n(G)-\alpha'(G).$$\end{theorem}

It is known that for every graph
$G$ with no isolated vertex $\alpha'(G)\ge \gamma(G)$ \cite{Haynes1998}. Hence,  Theorem \ref{alphaCockayne2003} leads to the following corollary.

\begin{corollary}\label{SecuredominationVSdomination}
If a graph $G$ does not have isolated vertices, then 
$$\gamma_s(G)\le n(G)-\gamma(G).$$
\end{corollary}

Recall that a graph without isolated vertices satisfies $\gamma(G)=n(G)/2$ if and only if its components are isomorphic to $C_4$ or to corona graphs of the form $H\odot K_1$.
If $\gamma(G)=n(G)/2$, then Corollary \ref{SecuredominationVSdomination} leads to 
$\frac{n(G)}{2}=\gamma(G)\le \gamma_r(G)\le \gamma_s(G)\le \frac{n(G)}{2}$. Thus, we deduce
the following result.

\begin{remark}\label{DominationAndSecure=half}
 If $\gamma(G)=\frac{n(G)}{2}$, then $\gamma_r(G)=\gamma_s(G)=\frac{n(G)}{2}.$
\end{remark}

As we will show in Theorem \ref{SecuredominationVSdominationTau}, in some cases the bound provided by Theorem \ref{alphaCockayne2003} can be improved. To this end, we need to introduce some additional notation. Let $\mathcal{D}(G)$ be the set of all $\gamma(G)$-sets. For every $S\in \mathcal{D}(G)$ we define 
$$T(S)=\{v\in V(G)\setminus S:\; N[v]=N[s] \text{ for some } s\in S\}.$$
Finally, we define 
$$\tau(G)=\max \{|T(S)|:\; S\in \mathcal{D}(G)\}.$$
Recall that two vertices $u,v$ are called \emph{true twins} if $N[u]=N[v] $.

\begin{lemma}\label{securecomplementlesstau}
Let $G$ be a graph such that no component of $G$ is a complete graph. If $S$ is a  $\gamma(G)$-set, then $V(G)\setminus (S\cup T(S))$ is a dominating set.
\end{lemma}

\begin{proof}
Since every vertex in $T(S)$ has a true twin in $S$, we only need to show that every vertex in $S$ has a neighbour in $S'=V(G)\setminus (S\cup T(S))$. 

Notice that,  since $G$ has no isolated vertices and $S$ is a $\gamma(G)$-set, every vertex in $S$ has at least one neighbour outside of $S$.
Suppose that there exists $s\in S$ such that $N(s)\cap S'=\emptyset$.  In such a case, $N(s)\cap T(S)\ne \emptyset$ and, if $N(s)\cap S=\emptyset$, then the subgraph induced by $N[s]$ is a component of $G$, which is a contradiction. Thus, $N(s)\cap S\ne \emptyset$.
Now, let $x\in N(s)\cap T(S)$.
If $s$ and $x$ are true twins, then every neighbour of $s$ belonging to $S$ is a neighbour of $x$, while if $s$ and $x$ are not true twins, then there exists $s''\in S\setminus \{s\}$ which is twin with $x$. Therefore, 
 $S\setminus \{s\}$ is a dominating set, which is  a contradiction.  
\end{proof}

\begin{theorem}\label{SecuredominationVSdominationTau}
 If no component of $G$ is a complete graph, then
$$\gamma_s(G)\le n(G)-\gamma(G)-\tau(G).$$
\end{theorem}

\begin{proof}
Let $S$ be  a $\gamma(G)$-set such that $|T(S)|=\tau(G)$.  We will show that $S'=V(G)\setminus (S\cup T(S))$ is a secure dominating set.  
We already know from Lemma \ref{securecomplementlesstau}  that $S'$ is a dominating set. It remains to show that for every $v\in S\cup T(S)$ there exists $u\in S'\cap N(v)$ such that $S'_{uv}=(S'\setminus \{u\})\cup \{v\}$ is a dominating set. To this end, 
for every $u\in S'$ we define $P(u)$ as follows:
 $$P(u)=\{v\in S:\; N(v)\cap S'=\{u\}\}.$$
 If there exists $u\in S'$ such that $|P(u)|\ge 2$, then $S_1=(S\setminus P(u))\cup \{u\}$ is a dominating set and $|S_1|<|S|=\gamma(G)$, which is a contradiction. Hence, $|P(u)|\le 1$ for every $u\in S'$. With this fact in mind, we differentiate two cases for $v\in V(G)\setminus S'$. 

\vspace{0.3cm}
\noindent Case 1: $v\in S$.  Suppose that $P(u)=\{v\}$ for some $u\in S'$. In this case, for every $w\in N(u)\cap (S\setminus \{v\})$ we have  $|N(w)\cap S'|\ge 2$. So that, if there exists $y\in \left(N(u)\cap T(S)\right) \setminus N(v)$, then $|N(y)\cap S'|\ge 2$, as $y$ has a twin in $S\setminus \{v\}$.
Hence,  $S'_{uv}$ is a dominating set.  From now on we assume that 
$|N(v)\cap S'|\ge 2$. Now, if there exists 
 $u'\in N(v)\cap S'$ such that 
$P(u')=\emptyset$, then $|N(w)\cap S'|\ge 2$ for every $w\in N(u')\cap (S\setminus \{v\})$, and also for every $w\in \left(N(u')\cap T(S)\right) \setminus N(v)$, which implies  $S'_{u'v}$ is a dominating set. 
 Finally, suppose that $P(u)\ne \emptyset$ for every $u\in N(v)\cap S'$. Let
 $$ X=\{v\}\cup\left( \bigcup_{u\in N(v)\cap S'}P(u)\right).$$
Notice that $|X|=1+|N(v)\cap S'|$. Hence, 
$S_2=(S\setminus X)\cup (N(v)\cap S')$ is a dominating set of $G$ and $|S_2|<|S|$, which is a contradiction.

\vspace{0.3cm}
\noindent Case 2: $v\in T(S)$. Let $v'\in S$ such that $N[v]=N[v']$. As discussed in Case 1, there exists $u\in S'$ such that $S'_{uv'}$ is a dominating set. Since $v$ and $v'$ are true twins, we can conclude that $S'_{uv}$  is also a dominating set.
 
 According to the two cases above, $S'$ is a secure dominating set. 
Therefore,   $\gamma_s(G)\le |S'|= n(G)-\gamma(G)-\tau(G)$.
\end{proof}

To show an example where Theorem \ref{SecuredominationVSdominationTau} improves the bound given by Theorem \ref{alphaCockayne2003}, we take the graph $G\cong K_3+N_2\cong K_5-e$. In this case $\gamma (G) =1$, $\tau (G)=2$ and $\alpha'(G)=2$, which implies that $\gamma_s(G)\le n(G)-\gamma (G)-\tau (G)=2<3=n(G)-\alpha'(G)$.

A set $X\subseteq V(G)$ is  called a $2$-\textit{packing}
if $N[u]\cap N[v]=\emptyset $ for every pair of different vertices $u,v\in X$.
 The
$2$-\textit{packing number} $\rho(G)$ is the   cardinality
of any   largest $2$-packing of
$G$. A $2$-packing of cardinality $\rho(G)$ is called a $\rho(G)$-set.
It is well known that for any graph $G$, $\gamma(G)\ge \rho(G)$,   \cite{Haynes1998}.
Meir and Moon \cite{MR0401519} showed in 1975 that $\gamma(T)= \rho(T)$ for any tree $T$.  We remark that in general, these $\gamma(T)$-sets and $\rho(T)$-sets  are not identical. The following result is a direct consequence of Theorem \ref{SecuredominationVSdominationTau}.

\begin{corollary}\label{SecuredominationVS2-packing}
 If no component of $G$ is a complete graph, then $$\gamma_s(G)\le n(G)-\rho(G)-\tau(G).$$
\end{corollary}

To see the sharpness of the bound above, consider the corona graph $G_1\odot N_p$, where $G_1$ is an arbitrary graph. In this case, $n(G_1\odot N_p)=n(G_1)(p+1)$,  $\rho(G_1\odot N_p)=n(G_1)$ and  $\gamma_s(G_1\odot N_p)=n(G_1)p=n(G_1\odot N_p)-\rho(G_1\odot N_p)=n(G_1\odot N_p)-\gamma(G_1\odot N_p)$. From $G'\cong G_1\odot N_2$
we can construct a family of graphs $G$ of order $n(G)=3n(G_1)+l_1+\dots +l_{n(G_1)}$ with $\gamma_s(G)=n(G)-\gamma(G)-\tau(G)$. We construct $G$ from $G'$ and a $\gamma(G')$-set $S=\{v_1,\dots, v_{n(G_1)}\}$ by replacing every $v_j\in S$ with a copy of $K_{l_j}$ and joining by an edge each vertex of  $K_{l_j}$ with each neighbour of $v_j$ in $G'$. 

As shown in \cite{Walikar1979}, the domination number of any graph $G$ is bounded below by $\frac{n(G)}{\Delta(G)+1}$. Therefore, the following result is deduced from Theorem~\ref{SecuredominationVSdominationTau}.

\begin{corollary}\label{SecuredominationVS2-packing}
  If no component of $G$ is a complete graph,  then $$\gamma_s(G)\le \left\lfloor \frac{n(G)\Delta(G)}{\Delta(G) +1} \right\rfloor-\tau(G).$$
\end{corollary}

The bound above is tight. For instance, it is achieved for any graph isomorphic to $K_n-e$. In this case $\tau(G)=n(G)-3$ and $\Delta(G)=n(G)-1$ so $\gamma_s(G)=2$.

Since $\gamma_r(G)\le 2\gamma(G)$ and $\gamma_r(G)\le  \gamma_s(G)$, Theorem \ref{SecuredominationVSdominationTau} leads to the following upper bounds on the weak Roman domination number.

\begin{corollary}\label{SecuredominationVOrderDomination}
  If no component of $G$ is a complete graph, then the following assertions hold.
\begin{enumerate}[{\rm (i)}]
\item $\gamma_r(G)\le \displaystyle\left\lfloor \frac{n(G)+\gamma(G)-\tau(G)}{2} \right\rfloor$.
\item If $\gamma(G)\ge \frac{n(G)}{3}$, then $\gamma_r(G)\le 2\gamma(G)-\tau(G)$.
\end{enumerate}
\end{corollary}

To see the sharpness of the bounds above, consider the corona graph $G\cong G_1\odot N_2$, where $G_1$ is an arbitrary graph.    In this case,  $n(G)=3n(G_1)$,   $\gamma(G)=n(G_1)$, $\tau(G)=0$ and   $\gamma_r(G)=2n(G_1)$. Another example of equality for bound (i) is $G\cong K_n-e$, where $\gamma_r(G)=2$, $\tau(G)=n(G)-3$ and $\gamma(G)=1$.

The minimum number of cliques of a given graph $G $ needed to cover the vertex set $V(G)$  is called the \emph{clique covering number}  of  $G$ and denoted by $\theta (G)$. Before stating our next result we need to recall the following theorem, which states a Nordhaus-Gaddum inequality for the chromatic number of a graph. 

\begin{theorem}{\rm \cite{Borowiecki1976}}\label{NordHausGaddumChromatic}
For any graph $G$,
$$\chi(G)+\chi(\overline{G})\le n(G)+1\text{ and }\chi(G)\chi(\overline{G})\le\displaystyle\frac{(n(G)+1)^2}{4}.$$
\end{theorem}

\begin{theorem}
 The following statements hold for any graph $G$.
\begin{enumerate}[{\rm (i)}]
\item $\gamma_s(G)\le \theta(G).$
\item $\gamma_r(G)+\gamma_r(\overline{G})\le \gamma_s(G)+\gamma_s(\overline{G})\le n(G)+1.$
\item $\gamma_r(G)\gamma_r(\overline{G})\le \gamma_s(G)\gamma_s(\overline{G})\le\displaystyle\frac{(n(G)+1)^2}{4}.$\end{enumerate}
Furthermore, if $G\not\cong C_5$ is a connected graph with $\delta(G)\ge 2$ and $\Delta(G) \le n(G)-3$, then  the following  statement hold. 
\begin{enumerate}[{\rm (iv)}]
\item
$\gamma_r(G)+\gamma_r(\overline{G})\le \gamma_s(G)+\gamma_s(\overline{G})\le n(G)-1$ for $n(G)$ odd and 

 $\gamma_r(G)+\gamma_r(\overline{G})\le \gamma_s(G)+\gamma_s(\overline{G})\le n(G)$ for $n(G)$ even.
\item $\gamma_r(G)\gamma_r(\overline{G})\le \gamma_s(G)\gamma_s(\overline{G})\le \frac{(n(G)-1)^2}{4}$ for $n(G)$ odd and  

$\gamma_r(G)\gamma_r(\overline{G})\le \gamma_s(G)\gamma_s(\overline{G})\le \frac{(n(G))^2}{4}$ for $n(G)$ even.
\end{enumerate}
\end{theorem}

\begin{proof}
Let $\Pi $ be a partition of $V(G)$ into cliques such that $|\Pi|=\theta(G)$. 
The proof of (i) directly follows from the fact that any set formed by one representative of each clique in $\Pi$ is a secure dominating set.  

Since  $\chi(G)=\theta(\overline{G})$, (i) and Theorem \ref{NordHausGaddumChromatic} lead to $$\gamma_s(G)+\gamma_s(\overline{G})\le \theta(G)+\theta(\overline{G})=\chi(\overline{G}) + \chi(G)\le n(G)+1$$ and $$\gamma_s(G)\gamma_s(\overline{G})\le \theta(G)\theta(\overline{G})=\chi(\overline{G})\chi(G)\le\displaystyle\frac{(n(G)+1)^2}{4},$$as required. Finally, (iv) and (v) are a direct consequence of Theorem \ref{UpperBoundnOver2}.
\end{proof}

The inequalities above are tight. For instance, (i) is achieved by the graphs 
shown in Figure \ref{selfcmplementary},  (ii) and (iii) are achieved by the self-complementary graph shown in Figure \ref{selfcmplementary} (on the left) and also by $C_5$. In both cases we have $n(G)=5$ and $\gamma_r(G)=\gamma_s(G)=3$. Finally, (iv) and (v) are achieved by the self-complementary graph shown in  Figure \ref{selfcmplementary} (on the right), in this case we have $n(G)=8$ and $\gamma_r(G)=\gamma_s(G)=4$.

\begin{figure}[h]\label{selfcmplementary}
\begin{center}
\begin{tikzpicture} [transform shape, inner sep = .7mm]
\node[draw=black, shape=circle, fill=black] (n11) at (-8,0.5) [thick] {};
\node[draw=black, shape=circle, fill=black] (n21) at (-7,-0.5) [thick] {};
\node[draw=black, shape=circle, fill=black] (n31) at (-6,0.5) [thick] {};
\node[draw=black, shape=circle, fill=black] (n41) at (-5,-0.5) [thick] {};
\node[draw=black, shape=circle, fill=black] (n51) at (-4,0.5) [thick] {};
\draw (n11)--(n21)--(n31)--(n41)--(n51);  \draw (n21)--(n41); 

\node[draw=black, shape=circle, fill=black] (n1) at (1,1) [thick] {};
\node[draw=black, shape=circle, fill=black] (n2) at (-1,1) [thick] {};
\node[draw=black, shape=circle, fill=black] (n3) at (-1,-1) [thick] {};
\node[draw=black, shape=circle, fill=black] (n4) at (1,-1) [thick] {};
\node[draw=black, shape=circle, fill=black] (n5) at (0,2) [thick] {};
\node[draw=black, shape=circle, fill=black] (n6) at (-2,0) [thick] {};
\node[draw=black, shape=circle, fill=black] (n7) at (0,-2) [thick] {};
\node[draw=black, shape=circle, fill=black] (n8) at (2,0) [thick] {};
\draw (n1)--(n2)--(n3)--(n4)--(n1)--(n5)--(n2)--(n4)--(n8)--(n1)--(n3)--(n7)--(n4); \draw (n2)--(n6)--(n3);
\end{tikzpicture}
\caption{Two self-complementary graphs.}
\end{center}
\end{figure}
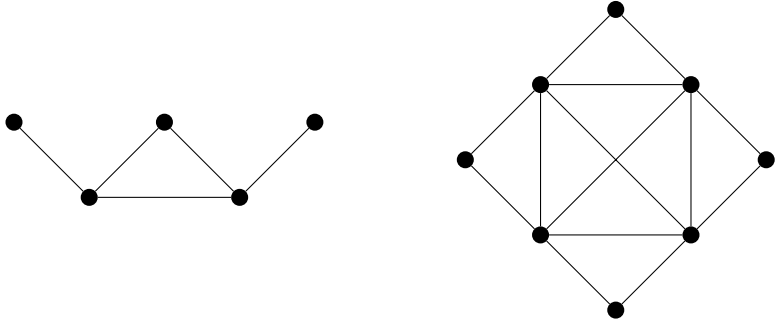

\section{Results on   Cartesian product   graphs}\label{SectionCartesianProduct}

The \textit{Cartesian product} of two graphs $G$ and $H$ is the graph $G\Box H$, such that $V(G\Box H)=V(G)\times V(H)$ and two vertices $(g,h),(g',h')\in V(G\Box H)$ are adjacent in $G\Box H$ if and only if either
\begin{itemize}
\item $g=g'$ and $hh'\in E(H)$, or
\item $gg'\in E(G)$  and $h=h'$.
\end{itemize}

The Cartesian product is a straightforward and natural construction, and is in many respects the simplest graph product \cite{Hammack2011,Imrich2000}. Hypercubes, Hamming graphs, grid graphs, cylinder graphs and torus graphs are some particular cases of this product. The \emph{Hamming graph} $H_{k,t}$ is the Cartesian product of $k$ copies of the complete graph $K_t$. 
The \emph{hypercube} $Q_t$ is defined as $H_{t,2}$. Moreover, the \emph{grid graph} $P_k\Box P_t$ is the Cartesian product of the paths $P_k$ and $P_t$, the \emph{cylinder graph} $C_k\Box P_t$ is the Cartesian product of the cycle $C_k$ and the path $P_t$, and the \emph{torus graph} $C_k\Box C_t$ is the Cartesian product of the cycles $C_k$ and $C_t$.

This operation is commutative  in the sense that $G\Box H \cong H\Box G$, and is also associative, as the graphs $(F \Box G)\Box H$ and $F\Box (G\Box H)$ are naturally isomorphic. A Cartesian product  graph is connected if and only if both of its factors are connected.

Notice that  for any $u\in V(G)$  and $v\in V(H)$  the subgraph of $G\Box H$ induced by $\{u\}\times V(H)$ is isomorphic to $H$ and the subgraph of $G\Box H$ induced by $V(G)\times \{v\}$ is isomorphic to $G$.

This product has been extensively investigated from various perspectives. For instance, the most popular open problem in the area of domination theory is known as Vizing's conjecture. Vizing \cite{Vizing1968} suggested that for any graphs $G$ and $H$,
$$\gamma(G\Box H)\ge \gamma(G)\gamma(H).$$
Several researchers have worked on it, for instance, some partial results appears in \cite{Bresar2012,Hammack2011}.  For more information on structure and properties of the Cartesian product of graphs we refer the reader to \cite{Hammack2011,Imrich2000}.

The study of the secure domination number of Cartesian product graphs was initiated by Cockay\-ne et al.\ in \cite{MR2137919}, where  they obtained bounds on $\gamma_s(C_k\Box C_{t})$ and $\gamma_s(P_k\Box P_{t})$ in terms of $k$ and $t$. 
Before stating our first result we need to recall the following well known lower bound on the domination number of any Cartesian product graph. 

\begin{lemma}
{\rm \cite{MR1110240}}\label{LowerBoundDominationCartesian} For any pair of graphs $G$ and $H$, 
$$\gamma(G\Box H)\ge \min\{n(G),n(H)\}.$$
\end{lemma}

\begin{theorem}\label{MainUpperBoundWRD-Cartesian}
For any graphs  $G$ and $H$, the following statements hold.
\begin{enumerate}[{\rm (i)}]
\item $\min\{n(G),n(H)\}\le \gamma_r(G \Box H)\le \min \{n(G)\gamma_r(H),n(H)\gamma_r(G)\}.$

\item $\min\{n(G),n(H)\}\le \gamma_s(G \Box H)\le \min \{n(G)\gamma_s(H),n(H)\gamma_s(G)\}.$
\end{enumerate} 

\end{theorem}

\begin{proof}
Let $f(U_0,U_1,U_2)$ be a $\gamma_r(G)$-function. In order to prove the upper bound, we claim that the function $g:V(G\Box H)\longrightarrow \{0,1,2\}$ defined by 
$g(x,y)=f(x)$ is a WRDF on $G\Box H$, where $$\{W_0=U_0\times V(H),W_1=U_1\times V(H),W_2=U_2\times V(H)\}$$ is the partition of $V(G\Box H)$ associated to $g$. To see this we only need to observe the following two facts.

\noindent Fact (a): Since every $x\in U_0$ is dominated by some $x'\in U_1\cup U_2$, every $(x,y)\in W_0$ is dominated by $(x',y)\in W_1\cup W_2$. 
\\
\noindent Fact (b): Since for every $x\in U_0$ there exists  $x'\in N(x)\cap (U_1\cup U_2)$ such that the movement of a guard from $x'$ to $x$ does not produce undefended vertices in $G$,    the movement of a guard from  $(x',y)\in W_1\cup W_2$  to $(x,y)\in W_0$ does not produce undefended vertices in the subgraph of $G\Box H$ induced by $V(G)\times \{y\}$, which is isomorphic to $G$.

According to Facts (a) and (b) we can conclude that $g$ is a WRDF on $G\Box H$, which implies that
$\gamma_r(G \Box H)\le w(g)=n(H)w(f)=n(H)\gamma_r(G),$ as required. By analogy we deduce that $\gamma_r(G \Box H)\le n(G)\gamma_r(H).$ Therefore, the upper bound of (i) follows. The proof of the upper bound of (ii) is deduced by analogy to the previous one  by taking a   WRDF  $f(U_0,U_1,U_2)$ such that  $U_2=\emptyset$ and $|U_1|=\gamma_s(G)$.

Finally, 
the lower bounds are deduced from Lemma \ref{LowerBoundDominationCartesian}, as $\gamma_s(G \Box H)\ge \gamma_r(G \Box H)\ge \gamma(G \Box H)\ge  \min\{ n(G),n(H)\}.$ 
\end{proof}

As we well show in the following results, the bounds above are tight.

\begin{corollary}
\label{WRDNumberCartesianCompleteTimesH}
Let $t$ be an integer. If  $2 \le n(H)\le t $,  then $ \gamma_r(K_t \Box H)=\gamma_s(K_t \Box H)=n(H).$
\end{corollary}

According to the result above, it remains to study the weak Roman domination number and the secure domination number of $K_t \Box H$ for $n(H)>t$. Our next result covers two particular cases.

\begin{proposition}
For any integers $t\ge 3$ and $t'\ge 3$, 
$$\gamma_r(K_t\Box C_{t'})=\gamma_r(K_t\Box P_{t'})=\gamma_s(K_t\Box P_{t'})=\gamma_s(K_t\Box C_{t'})=t'.$$
\end{proposition}

\begin{proof}
By Theorem \ref{MainUpperBoundWRD-Cartesian} and Propositions  \ref{DominationChain1} and \ref{RemovingEdges}   we have that $$\min\{t,t'\}\le\gamma_r(K_t\Box C_{t'}) \le \gamma_r(K_t\Box P_{t'})\le t'$$ and $$\min\{t,t'\}\le \gamma_r(K_t\Box C_{t'}) \le \gamma_s(K_t\Box C_{t'})\le   \gamma_s(K_t\Box P_{t'})\le t'.$$ It remains to show that  $\gamma_r(K_t\Box C_{t'})\ge t'$ for $t' > t\ge3$.   Let $f(W_0,W_1,W_2)$ be a  $\gamma_r(K_t\Box C_{t'})$-function and  $V(C_{t'})=\{v_1,\dots , v_{t'}\}$, where the subscripts are taken modulo $t'$ and $v_iv_{i+1}\in E(C_{t'})$ for any $i\le t'$. Let $A_i=(V(K_t)\times \{v_i\})$ and $\alpha_i=f(A_i)$  for every $i\in \{1,\dots, t'\}$.
We differentiate the following cases in which $\alpha_i=0$ for some $i$. Symmetric cases are omitted.
\\
\\
\noindent Case 1: $\alpha_i=0$. Since $W_1\cup W_2$ is a dominating set, we can conclude that 
$$\alpha_{i-1}+\alpha_i+\alpha_{i+1}\ge t\ge 3.$$

\noindent Case 2:  $\alpha_{i-1}=\alpha_{i+1}=0$ and  $\alpha_{i}=1$. In this case, no guard can move from $A_i$ to
$A_{i+1}$ (or to $A_{i-1}$), which implies that $\alpha_{i-2}\ge t$ and $\alpha_{i+2}\ge t$. Hence,  we can conclude that 
$$\alpha_{i-2}+\alpha_{i-1}+\alpha_{i}+\alpha_{i+1}\ge t+1\ge 4 \text{ and } \alpha_{i-1}+\alpha_{i}+\alpha_{i+1}+\alpha_{i+2}\ge 1+t\ge 4. $$
In this case, if $t'\ge 6$, then 
$$\alpha_{i-2}+\alpha_{i-1}+\alpha_{i}+\alpha_{i+1}+\alpha_{i+2}\ge 2t+1\ge 7. $$

\noindent Case 3: $\alpha_{i}= 2$ and  $\alpha_{i-1}=\alpha_{i+1}=0$. From Case 1 we know that $\alpha_{i-2}\ge t-2$ and $\alpha_{i+2}\ge t-2$. Suppose that $\alpha_{i-2}= t-2$ and $\alpha_{i+2}< t$.
Notice  that  $W_2\cap (A_{i-2}\cup A_{i})=\emptyset$, as every vertex in $A_{i-1}$ has to be dominated by some vertex in $W_1\cup W_2$. Hence, for $(u,v_i),(u',v_i)\in V_1$ we have  that $(u,v_{i-2}),(u',v_{i-2})\in V_0$ and  $(u,v_{i+2})\in V_0$ or $(u',v_{i+2})\in V_0$, as $\alpha_{i-2}= t-2$ and $\alpha_{i+2}<t.$ We can assume that $(u,v_{i+2})\in V_0$. Thus, the movement of a guard form $(u,v_i)$  to $(u,v_{i-1})$  produces undefended vertices in $A_{i+1}$,  which is a contradiction. Hence, $\alpha_{i-2}+\alpha_{i+2}\ge 2(t-1)$ and so we can conclude that 
$$\alpha_{i-2}+\alpha_{i-1}+\alpha_{i}+\alpha_{i+1}+\alpha_{i+2}\ge 2t\ge 6.$$

According  to the conclusions derived from the cases above we can deduce that, 
$$\gamma_r(K_t\Box C_{t'})=w(f)=
\sum_{i=1}^{t'}\alpha_{i}\ge t'.$$ 
 Therefore, the result follows.
\end{proof}
Notice that the result above does not include the case of complete graphs of order two. For this case we propose the following conjecture.

\begin{conjecture}
For any integer $t\ge 2$  
$$\gamma_s(P_t\Box K_2)=\left\lceil\frac{3t+1}{4}\right\rceil.$$
Furthermore, for  $t\ge 3$,
$$\gamma_s(C_t\Box K_2)= \left\lbrace \begin{array}{ll}
\left\lceil\frac{3t}{4}\right\rceil +1, & \text{ if } t\equiv 4\text{ \rm (mod } 8) 
\\
\\
\left\lceil\frac{3t}{4}\right\rceil, & \text{ otherwise. } 
\end{array}\right.
$$
\end{conjecture}
 
 Regarding the conjecture above, we would emphasize that it is known from 
\cite{CartProdCycles} that $\gamma(P_t\Box K_2)=\left\lceil \frac{t+1}{2}\right\rceil$  and from  \cite{Cockayne2004} that $\gamma_R(P_t\Box K_2)=t+1$.
 
\begin{proposition}\label{PropositionCompleteCartesianStars}
Let $t\ge 2$ and   $ t'\ge 2$ be two integers. The following statements hold.
\begin{enumerate}[{\rm (i)}]
\item  $\gamma_r(K_t \Box K_{1,t'-1})=\min\{2t,t'\}$.

\item  $\gamma_s(K_t \Box K_{1,t'-1})=t'$.
\end{enumerate}
\end{proposition}

\begin{proof}
 From Theorem \ref{MainUpperBoundWRD-Cartesian} we have that $\gamma_r(K_t \Box K_{1,t'-1})\le \min\{2t,t'\}.$  We proceed to show that  $\gamma_r(K_t \Box K_{1,t'-1})\ge \min\{2t,t'\}$.
 Let $f(W_0,W_1,W_2)$ be a $\gamma_r(K_t \Box K_{1,t'-1})$-function and let $y_0$ be the universal vertex of $K_{1,t'-1}$.   
  Suppose that $\gamma_r(K_t \Box K_{1,t'-1})< \min\{2t,t'\}$. Now, since $\gamma_r(K_t \Box K_{1,t'-1})<2t$, there exists $x\in V(K_{t})$ such that $f(\{x\}\times V(K_{1,t'-1}))\le 1$ and, since $\gamma_r(K_t \Box K_{1,t'-1})<t'$,  there exist $y\in V(K_{1,t'-1})$  such that    $ V(K_{t})\times\{y\}\subseteq W_0$. 
If $y=y_0$, then there is exactly one guard for each copy of $K_t$ different from the one associated to $y_0$ (as every vertex has to be defended), which implies that the movement of any guard to a vertex in $V(K_t)\times \{y_0\}$ produces undefended vertices, so that $y\ne y_0$.
Notice that $f( V(K_{t})\times\{y_0\})\ge t$, otherwise there are undefended vertices in $ V(K_{t})\times\{y\}$. 
 Now, suppose that $V(K_{t})\times\{y'\} \subseteq W_0$, for some $y'\in V(K_{1,t'-1})\setminus \{y_0,y\}$. In such a case,  $(x,y')$ and 
 $(x,y)$ are only defended by a guard located at $(x,y_0)$, but $(x,y)$ will become undefended after the movement of that guard to $(x,y')$, which is a contradiction.  Hence, $\sum_{v\ne y_0}f( V(K_{t})\times\{v\})\ge t'-2$, and so 
 $w(f)\ge t+t'-2\ge t'$, which is a contradiction again. Thus, $\gamma_r(K_t \Box K_{1,t'-1})\ge \min\{2t,t'\}$, as required. 
 Therefore, (i)  follows. 
 
  We now proceed to prove (ii). As above, let $y_0$ be the universal vertex of $K_{1,t'-1}$,  $W$  a $\gamma_s(K_t \Box K_{1,t'-1})$-set and $u\in V(K_t)$. Suppose that $|W|\le t'-1$. In such a case, there exists $v\in V(K_{1,t'-1})$ such that $W\cap (V(K_t)\times \{v\})=\emptyset$. Notice that $N(u,v)\cap W\ne \emptyset$. We differentiate two cases. 
  \\
  \noindent Case 1:  $v\ne y_0$. Since $W$ is a dominating set,  $V(K_t)\times \{y_0\}\subseteq W$. Thus, there exists $v_1\in V(K_{1,t'-1})\setminus\{v,y_0\}$ such that $V(K_t)\times \{v_1\}\subseteq \overline{W}$. Hence,   $N[(u,v)]\cap W=\{(u,y_0)\}=N[(u,v_1)]\cap W$, and so $(W\setminus \{(u,y_0)\})\cup \{(u,v_1)\}$ is not a dominating set, which is a contradiction. 
    \\
  \noindent Case 2:  $v=y_0$. Since $W$ is a dominating set and $|W|<t'$, for every $v'\in V(K_{1,t'-1})\setminus\{y_0\}$ we have that 
 $|(V(K_t)\times \{v'\})\cap W|=1$. 
 Hence,  for  $u\in V(K_t)$ such that $(u,v')\in W$ and $u'\in V(K_t)\setminus \{u\}$ we have that   $N[(u',v')]\cap W=\{(u,v')\}$. Thus, for every $v'\in V(K_{1,t'-1})\setminus\{y_0\}$ and $u\in V(K_t)$ such that $(u,v')\in W$, we have that $(W\setminus \{(u,v')\})\cup \{(u,y_0)\}$ is not a dominating set, which is a contradiction.
 \\
  According to the two cases above we can conclude that $\gamma_s(K_t \Box K_{1,t'-1})=|W|\ge t'$. Finally,  Theorem 
 \ref{MainUpperBoundWRD-Cartesian} leads to $\gamma_s(K_t \Box K_{1,t'-1})=t'$. \end{proof}

\begin{proposition}
For any   graph $G$ and any integer $t>2n(G)\ge 4$,
$$\gamma_r(G \Box K_{1,t-1})=2n(G).$$
\end{proposition}

\begin{proof}
By Theorem 
 \ref{MainUpperBoundWRD-Cartesian} we have $\gamma_r(G \Box K_{1,t-1})\le 2n(G).$ To conclude the proof we only need to observe that Propositions \ref{RemovingEdges} and  \ref{PropositionCompleteCartesianStars}  lead to $\gamma_r(G \Box K_{1,t-1})\ge \gamma_r(K_{n(G)} \Box K_{1,t-1})= 2n(G).$ 
\end{proof}

\begin{theorem}\label{SecuredominationCartesian}
 If no component of a graph $H$ is a complete graph, then for any nontrivial graph $G$,
$$  \gamma_s(G\Box H)\le n(G)\gamma(H)+n(H)\gamma(G)- 2\gamma(G)\gamma(H)-\gamma(G)\tau(H).$$ 
\end{theorem}

\begin{proof}
In this proof we use the set $T(S)$  as defined prior to Lemma  \ref{securecomplementlesstau}. Let $S_1$  be a $\gamma(G)$-set and $S_2$ a  $\gamma(H)$-set such that $|T(S_2)|=\tau(H)$. We will show that   $W=(S_1\times S'_2) \cup (\overline{S}_1\times  S_2)$ is a secure dominating set of $G\Box H$, where $S_2'=V(H)\setminus (S_2\cup T(S_2))$. First of all, notice that $W$ is a dominating set of $G\Box H$ as $S_2$ and $S'_2$ are dominating sets in $H$ (by Lemma \ref{securecomplementlesstau}). We differentiate the following three cases for $(x,y)\in \overline{W}$.

\noindent Case 1: $(x,y)\in S_1\times \overline{S'_2}$. In the proof of Theorem  \ref{SecuredominationVSdominationTau} we have shown that $S'_2$ is a secure dominating set. Hence, for each vertex $(x,y)\in S_1\times \overline{S'_2}$ there exists $(x,y')\in S_1\times S'_2$ such that the movement of a guard from $(x,y')$ to $(x,y)$ does not produce undefended vertices in $\{x\}\times\overline{S'_2}$. Such a movement of guards does not produce undefended vertices in $\overline{S_1}\times \{y'\}$, as these vertices are dominated by the ones in $\overline{S_1}\times S_2$.

\noindent Case 2:  $(x,y)\in \overline{S_1}\times S'_2$. For any $y'\in S_2\cap N(y)$ the movement of a guard from $(x,y')$ to $(x,y)$ does not produce undefended vertices in $S_1\times \{y'\}$, as these vertices are dominated by the ones in $S_1\times \{y\}$. Such a movement of guards does not produce undefended vertices in $\{x\}\times S'_2$, as these vertices are dominated by the ones in $\{x'\}\times S'_2$, for every $x'\in S_1\cap N(x)$. Now, suppose that  $y''\in N(y')\cap T(S_2)$.  If $|N(y'')\cap S_2|\ge 2$, then $(x,y'')$ remains defended after the above mentioned movement of guards. If $|N(y'')\cap S_2|=\{y'\}$, then $y'$ and $y''$ are twins, which implies $(x,y'')\in N(x,y)$, so that $(x,y'')$ remains defended after the movement of a guard form $(x,y')$ to $(x,y)$.

\noindent Case 3: $(x,y)\in \overline{S_1}\times T(S_2)$. Let $y'\in S_2$ such that $N[y]=N[y']$. As in the previous case, the movement of a guard from $(x,y')$ to $(x,y)$ does not produce undefended vertices in $S_1\times \{y'\}$. On the other hand, since $y$ and $y'$ are twins, the movement of a guard from $(x,y')$ to $(x,y)$ does not produce undefended vertices in $\{x\}\times\overline{S_2}$.

According to the three cases above, $W$ is a secure dominating set of $G\Box H$. Therefore, 
$$\gamma_s(G\Box H)\le |W|=n(G)\gamma(H)+n(H)\gamma(G)- 2\gamma(G)\gamma(H)-\gamma (G)\tau(H)$$
as desired.
\end{proof}

According to the result above, for any noncomplete graph $H$,
$$\gamma_s(K_t\Box H)\le (t-2)\gamma(H)+n(H)-\tau(H).$$
It is not difficult to check that the bound above is tight. For instance, it is achieved  by $H\cong K_l+N_3$ for $l\ge 2$, as $\gamma_s(K_3\Box (K_l+N_3))=5$, $\gamma(H)=1$ and $\tau(H)=l-1.$ Notice that, in this case, Theorem \ref{SecuredominationCartesian} gives a better result than
Theorem \ref{MainUpperBoundWRD-Cartesian}.

We learned from Theorem \ref{UpperBoundnOver2} that  $\gamma_s(G)\le \left\lfloor \frac{n(G)}{2}\right\rfloor $ for every graph $G\not\cong C_5$ having minimum degree $\delta(G)\ge 2$. 
If $G$ and $H$ have no isolated vertices, then $\gamma(G)\in \{1,\dots , \left\lfloor n(G)/2\right\rfloor \}$ and $\gamma(H)\in \{1,\dots , \left\lfloor n(H)/2\right\rfloor \}$. Hence, we can state the following remark which shows that the bound provide by  Theorem \ref{SecuredominationCartesian} is never worse that the bound $\gamma_s(G\Box H)\le  \left\lfloor\frac{n(G)n(H)}{2}\right\rfloor$ deduced from Theorem \ref{UpperBoundnOver2}.

\begin{remark}\label{MaximumUpperBoundSecuredominationCartesian}
 If $G$ and $H$ have no isolated vertices, then 
$$n(G)\gamma(H)+n(H)\gamma(G)- 2\gamma(G)\gamma(H)\le \left\lfloor\frac{n(G)n(H)}{2}\right\rfloor.$$ 
\end{remark}

The inequality chain $$\gamma_r(G\Box H) \le \gamma_s(G\Box H)\le n(G)\gamma(H)+n(H)\gamma(G)- 2\gamma(G)\gamma(H)$$  is tight. It is achieved for $P_3\Box P_3 $ and  $K_2\Box K_2\cong C_4$, as  $\gamma_r(P_3\Box P_3)=4$ and $\gamma_r(C_4)=2$.
Proposition \ref{CartesianStars} provides another example of graphs for which this inequality chain is achieved. 


\begin{proposition}\label{CartesianStars}
For any integer $t\ge 3$,
$$\gamma_r(K_{1,t-1}\Box K_{1,t-1})=\gamma_s(K_{1,t-1}\Box K_{1,t-1})=2(t-1).$$
\end{proposition}

\begin{proof}
According to Theorem \ref{SecuredominationCartesian}, we only need to prove the lower bound $\gamma_r(K_{1,t-1}\Box K_{1,t-1})\ge 2(t-1)$. Let $f(W_0,W_1,W_2)$ be a $\gamma_r(K_{1,t-1}\Box K_{1,t-1})$-function and, for simplicity, set $V=V(K_{1,t-1}) 
$. Let $x\in V$ be the vertex of  degree $t-1$. From now on, we suppose that $w(f)\le 2t-3$.  We proceed to show the following claim. 

\noindent
\textbf{Claim 1}.  $f(\{u\}\times V)\ge 1$, for every $u\in V\setminus \{x\}$.

In order to prove Claim 1, we suppose that  there exists $u\in V\setminus \{x\}$ such that $f(\{u\}\times V)=0$. In such a case, $f(x,y)\ge 1$, for every $y \in V$. Now, since  $w(f)\le 2t-3$, there exist $u'\in V\setminus \{x,u\}$ and $v\in V$ such that  $f(\{u'\}\times V)=0$ and $f(x,v)=1$, which is a contradiction as $(u',v)$ is undefended after the movement of the guard located in $(x,v)$ to $(u,v)$. Thus, Claim 1  follows. 

Since   $w(f)\le 2t-3$ ,   Claim 1 leads to the following ones.

\noindent
\textbf{Claim 2}.  There exists $u^*\in V\setminus \{x\}$ such that $f(\{u^*\}\times V)= 1$.   

\noindent
\textbf{Claim 3}.  There exists $v^*\in V\setminus \{x\}$ such that $f(x,v^*)= 0$.   

We differentiate the following two cases for $f(u^*,x)$.

\noindent Case 1: $f(u^*,x)=0$. By Claims 2 and 3 we can conclude that  $f(u^*,v^*)=1$, otherwise $(u^*,v^*)$  is not dominated by the elements in $W_1\cap W_2$. Since every vertex in $\{u^*\}\times V\setminus \{(u^*,x),(u^*,v^*)\}$ has to be dominated by some vertex in $W_1\cup W_2$, from   $w(f)\le 2t-3$ and Claim 1 we deduce that $f(x,v)=1$  for every $v\in V\setminus \{x,v^*\}$,   $f(\{u\}\times V)=1$  for every $u\in V\setminus \{x,u^*\}$, and $f(x,x)=0$. Hence, the movement of any guard from a vertex in $\{x\}\times V$ to $(x,x)$ produces undefended vertices in $\{u^*\}\times V$, and the movement of a guard from a vertex of the form $(a,x)$ to $(x,x)$ leaves vertex $(a,v^*)$ undefended.  In both cases we have a contradiction. 

\vspace{0,2cm}
\noindent  Case 2:  $f(u^*,x)=1$. In this case,  $(u^*,x)$ is the only vertex in $W_1\cup W_2$ which is adjacent to $(u^*,v^*)$. Hence, 
 the movement of a guard from $(u^*,x)$ to $(u^*,v^*)$ does not produce undefended vertices, and so from $w(f)\le 2t-3$ and Claim 1 we deduce that $f(x,v)=1$ for every $v\in V\setminus \{x,v^* \}$, $f(\{u\}\times V)=1$  for every $u\in V\setminus \{x,u^*\}$, and $f(x,x)=0$. Thus, the movement of a guard from a vertex of the form $(a,x)$ to $(x,x)$ leaves vertex $(a, v^*)$ undefended, which is a contradiction.

According to the two cases above we can conclude that, $w(f) \ge 2(t-1)$, as required.
\end{proof}


As usual in domination theory, when studying a domination parameter, we can ask if a Vizing-like conjecture can be proved or formulated. By Proposition~\ref{CartesianStars}  we can claim that there are graphs with $$\gamma_s(G\Box H)\not\ge \gamma_s(G)\gamma_s( H),$$ \textit{i.e.,}  for any $p\ge 3$ we have $\gamma_s(K_{1,p} \Box K_{1,p})=2p <p^2=\gamma_s(K_{1,p})\gamma_s(K_{1,p}).$ 

\begin{theorem}\label{AAAA}
Let $f_H=(V_0,V_1,V_2)$ be a $\gamma_r(H)$-function of a graph $H$ such that $V_2\ne \emptyset$, and let $Y= V(H)\setminus N[V_2]$. For any graph $G$, 
$$\gamma_r(G\Box H) \le 2n(G)|V_2| + |Y|\gamma_r(G).$$
\end{theorem}

\begin{proof}
Let $f_G=(U_0,U_1,U_2)$ be a $\gamma_r(G)$-function, $W_1 = U_1\times Y$ and $W_2 = (V(G)\times V_2) \cup (U_2\times Y)$. In order to show that  $f=(W_0,W_1,W_2)$ is a WRDF of $G\Box H$, we differentiate the following two cases for  $(x,y)\in W_0$.

\noindent Case 1: $(x,y)\in V(G)\times (N(V_2)\setminus V_2)$. Since  there exists $y'\in V_2\cap N(y)$, the movement of a guard from $(x,y')$ to $(x,y)$ does not produce undefended vertices.

\noindent Case 2:  $(x,y)\in U_0\times Y$. Since $f_G$ is a $\gamma_r(G)$-function, there exists $x'\in U_1\cup U_2$ such that the movement of a guard from $x'$ to $x$ does not  produce  undefended vertices. Which implies that the movement of a guard from $(x',y)$ to $(x,y)$ does not produce undefended vertices in  $V(G)\times Y$.
\end{proof}

Notice that for any graph with $\gamma_r(H)=2\gamma(H)$,  Theorems \ref{MainUpperBoundWRD-Cartesian} and \ref{AAAA} lead to the same result $\gamma_r(G\Box H) \le 2n(G)\gamma(H)$. In order to show an example where Theorem \ref{AAAA} gives a better result we take $G\cong K_3$ and the graph $H$ shown in Figure \ref{FigWeakRoman2}. In this case, an optimum solution consists of two guards at each vertex of the copy of $K_3$ corresponding to the vertex $v\in V(H)$ of maximum degree   and one guard at each copy of $K_3$ corresponding to the vertices of $H$ nonajacent to $v$. 

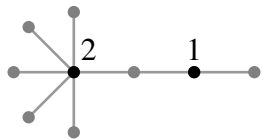
\begin{figure}[htb]
\begin{center}
\begin{tikzpicture}
[line width=1pt, scale=0.8]

\coordinate (V1) at (1,0);
\coordinate (V2) at (2,0);
\coordinate (V3) at (3,0);
\coordinate (V4) at (4,0);
\coordinate (V5) at (0.25,-0.75);
\coordinate (V6) at (0.25,0.75);
\coordinate (V7) at (1,-1);
\coordinate (V8) at (0,0);
\coordinate (V9) at (1,1);

\draw[black!40]  (V1)--(V2)--(V3)--(V4);
\draw[black!40]  (V1)--(V5);
\draw[black!40]  (V1)--(V6);
\draw[black!40]  (V1)--(V7);
\draw[black!40]  (V1)--(V8);
\draw[black!40]  (V1)--(V9);

\foreach \number in {1,...,9}{
\filldraw[gray]  (V\number) circle (0.08cm);
}
\foreach \number in {1,3}{
\filldraw[black]  (V\number) circle (0.08cm);
}

\node [above] at (1.25,0) {$2$};
\node [above] at (3,0) {$1$};
\end{tikzpicture}
\end{center}
\caption{A graph with $\gamma_r(H)=3$, $|Y|=2$ and $\gamma_r(K_3\Box H)=2n(G)|V_2| + |Y|\gamma_r(G)=8$. }
\label{FigWeakRoman2} 
\end{figure}

\vspace{1cc}
\noindent\textbf{Acknowledgment.}  This work has been partially supported by the Spanish Ministry of Economy, Industry and Competitiveness, under the grants MTM2016-78227-C2-1-P and MTM2017-90584-REDT.

\bibliographystyle{elsart-num-sort}

\end{document}